\documentclass[a4paper]{article}
\usepackage{RR,RRthemes}
\usepackage{amsmath,amsthm,amssymb,amscd
}
\usepackage{graphicx}
\usepackage{hyperref}

\usepackage[a4paper]{geometry}
\usepackage[utf8]{inputenc}

\usepackage[sc]{mathpazo}

\renewcommand{\epsilon}{\varepsilon}
\renewcommand{\phi}{\varphi}
 
\newcommand{\bR}{\mathbb{R}} \newcommand{\bN}{\mathbb{N}}
 
 \newcommand{\cH}{\mathcal{H}}

\newcommand{\cA}{\mathcal{A}}

\newcommand{\cB}{\mathcal{B}} 

\newcommand{\supp}{\operatorname{supp}}
\newcommand{\norm}[2]{\|#1\|_{#2}}
\newcommand{\seminorm}[2]{|#1|_{#2}}
\newcommand{\LRIP}{\operatorname{LRIP}}
\newcommand{\dict}{\mathbf{\Phi}}
\newtheorem{theorem}{Theorem}[section]
\newtheorem{lemma}[theorem]{Lemma}
\newtheorem{proposition}[theorem]{Proposition}

\theoremstyle{definition}

\newtheorem{example}[theorem]{Example}
\theoremstyle{remark}
\newtheorem{remark}[theorem]{Remark}
\numberwithin{equation}{section}

\newcommand{\Span}{\mathop{\rm span}}
\def\ripc{\delta}
\def\redund{R}

\def\slevel{\kappa}
\def\lfbound{A}
\def\ufbound{B}

\RRdate{February 2011}
\RRtitle{Sur la propri{\'e}t{\'e} d'isom{\'e}trie restreinte et l'approximation non-lin{\'e}aire avec un dictionnaire redondant}
\RRetitle{The restricted isometry property meets nonlinear\protect\\ approximation with redundant frames}
\title{The restricted isometry property meets nonlinear\protect\\ approximation with redundant frames}
\titlehead{The restricted isometry property meets nonlinear\protect\\ approximation with redundant frames}

\RRauthor{R{\'e}mi Gribonval
\thanks{R{\'e}mi Gribonval is with INRIA, Centre Inria Rennes - Bretagne Atlantique, Campus de Beaulieu, F-35042 Rennes Cedex, Rennes, France. Phone: +33 2 99 84 25 06. Fax: +33 2 99 84 71 71. Email: remi.gribonval@inria.fr.
}
\and Morten Nielsen
\thanks{Morten Nielsen is with the Department of Mathematical Sciences, Aalborg University, Frederik Bajersvej 7G, DK~-~9220 Aalborg East, Denmark. Email: mnielsen@math.aau.dk}
}

\authorhead{R. Gribonval and M. Nielsen}

\RRnote{This work was supported in part by the European Union through the project SMALL (Sparse Models, Algorithms and Learning for Large-Scale data). The project SMALL acknowledges the financial support of the Future and Emerging Technologies (FET) programme within the Seventh Framework Programme for Research of the European Commission, under FET-Open grant number: 225913.}
\RRnote{{\em NB: This work has been submitted for possible publication. Copyright may be transferred without notice, after which this version may no longer be accessible.}}

\RRabstract{It is now well known that sparse or compressible vectors can be stably recovered from their low-dimensional projection, provided the projection matrix satisfies a Restricted Isometry Property (RIP).
We establish new implications of the RIP with respect to nonlinear approximation in a Hilbert space with a redundant frame. The main ingredients of our approach are: a) Jackson and Bernstein inequalities, associated to the characterization of certain approximation spaces with interpolation spaces; b) a new proof that for overcomplete frames which satisfy a Bernstein inequality, these interpolation spaces are nothing but the collection of vectors admitting a representation in the dictionary with compressible coefficients; c) the proof that the RIP implies Bernstein inequalities.  As a result, we obtain that in most overcomplete random Gaussian dictionaries with fixed aspect ratio, just as in any orthonormal basis, the error of best $m$-term approximation of a vector decays at a certain rate if, and only if, the vector admits a compressible expansion in the dictionary. Yet, for mildly overcomplete dictionaries with a one-dimensional kernel, we give examples where the Bernstein inequality holds, but the same inequality fails for even the smallest perturbation of the dictionary.}

\RRresume{Il est maintenant bien {\'e}tabli que les vecteurs parcimonieux ou compressibles peuvent {\^e}tre estim{\'e}s de fa{\c c}on stable {\`a} partir de leur projection en petite dimension, d{\`es} que la matrice de projection satisfait une Propri{\'e}t{\'e} d'Isom{\'e}trie Restreinte (RIP). Nous {\'e}tablissons de nouvelles cons{\'e}quences de la RIP vis {\`a} vis de l'approximation non-lin{\'e}aire dans un espace de Hilbert avec un rep{\'e}re oblique (ou {\em frame}) redondant. 
Les principaux ingr{\'e}dients de notre approche sont : a) des in{\'e}galit{\'e}s de Jackson et de Bernstein, associ{\'e}es {\`a} des caract{\'e}risations de certains espaces d'approximation en termes d'espaces d'interpolation ; b) la preuve que pour des rep{\`e}res obliques satisfaisant ladite in{\'e}galit{\'e} de Bernstein, les espaces d'interpolation en question ne sont rien d'autres que l'ensemble des vecteurs ayant une repr{\'e}sentation {\`a} coefficients compressibles dans le dictionnaire ; c) la preuve que la RIP implique des in{\'e}galit{\'e}s de Bernstein. Une cons{\'e}quence de ces r{\'e}sultats est que la plupart des dictionnaires Gaussiens al{\'e}atoires de facteur de redondance fix{\'e} se comportent comme une base orthogonale du point de vue des espaces d'approximation : l'erreur optimale d'approximation {\`a} $m$ termes d'un vecteur d{\'e}cro{\^\i}t {\`a} une certaine vitesse en fonction de $m$ si, et seulement si, le vecteur admet une repr{\'e}sentation compressible dans le dictionnaire. Toutefois, nous montrons qu'il existe aussi des dictionnaires de redondance minimale, dont le noyau est de dimension un, pour lesquels l'in{\'e}galit{\'e} de Bernstein, bien que v{\'e}rifi{\'e}e, peut {\^e}tre mise en d{\'e}faut lorsque le dictionnaire subit une perturbation arbitrairement petite.}

\RRkeyword{Bernstein inequality, random dictionaries, restricted isometry condition} 
\RRmotcle{In{\'e}galit{\'e} de Bernstein, dictionnaire al{\'e}atoire, condition d'isom{\'e}trie restreinte} 

\RRprojet{METISS}
\RRthemeProj{metiss}
\RCRennes
\begin{document}
\RRNo{7548}
\makeRR
\allowdisplaybreaks

\section{Introduction}\label{sec:intro}
Data approximation using sparse linear expansions from overcomplete dictionaries has become a central theme in signal and image processing with applications ranging from data acquisition (compressed sensing) to denoising and compression. Dictionaries can be seen as collections of vectors $\{\phi_j\}$ from a Banach space $X$ equipped with a norm $\norm{\cdot}{X}$, and one wishes to approximate data vectors $f$ using $k$-term expansions $\sum_{j \in I} c_j \phi_j$ where $I$ is an index set of size $k$. Formally, using the matrix notation $\dict c = \sum_j c_j \phi_j$ and denoting $\norm{c}{0} = \sharp \{j, c_j \neq 0\}$ the number of nonzero components in the vector $c$ , we can define the (nonlinear) set of all such $k$-term expansions
\[
\Sigma_k(\dict) := \left\{\dict c,\quad \norm{c}{0} \leq k\right\}.
\]
\subsection{Best $k$-term approximation}
A first question we may want to answer, for each data vector $f$, is: {\em how well can we approximate it using elements of $\Sigma_k(\dict)$}? The error of best $k$-term approximation is a quantitative answer for a fixed $k$:
\[
\sigma_k(f,\dict) := \inf_{y \in \Sigma_k(\dict)} \norm{f-y}{X}.
\]
A more global view is given by the largest approximation rate $s>0$ such that\footnote{
The notation $a \lesssim b$ indicates the existence of a finite constant $C$ such that $a \leq C \cdot b$. The notation $a \asymp b$ means that we have both $a \lesssim b$ and $b \lesssim a$. As usual, $C$ will denote a generic finite constant, independent from the other quantities of interest. Different occurences of this notation in the paper may correspond to different values of the constant.}
\[
\sigma_k(f,\dict) \lesssim k^{-s},\quad \forall k \geq 1.
\]
To measure more finely the rate of approximation, one defines for $0<q<\infty$ \cite[Chapter 7, Section 9]{DeVore:1996aa}
\begin{equation}
\label{eq:DefApproxNorm}
\seminorm{f}{\cA_q^s(\dict)} := \left(\sum_{k \geq 1} \left[k^s \sigma_k(f,\dict)\right]^q k^{-1} \right)^{1/q} 
\asymp \left(\sum_{j \geq 0} \left[2^{js} \sigma_{2^j}(f,\dict)\right]^q\right)^{1/q}.
\end{equation}
and the associated {\em approximation spaces}
\begin{equation}
\label{eq:DefApproxSpace}
\cA_q^s(\dict) := \{f \in \cH, \seminorm{f}{\cA_q^s(\dict)}<\infty\}\quad 
\end{equation}
\subsection{Sparse or compressible representations}
Alternatively, we may be interested in {\em sparse / compressible representations} of $f$ in the dictionary. Suppose the vectors forming $\dict$ are quasi-normalized in $X$: for all $j$, $0< c \leq \norm{\phi_j}{X} \leq C < 
\infty$.
Then using  $\ell^\tau$ (quasi)-norms (in particular, $0 < \tau \leq 1$) one defines\footnote{It has been shown in~\cite{Gribonval:2004ab} that under mild assumptions on the dictionary, such as Eq.~\eqref{eq:DefUpperFrameBound}, the definition~\eqref{eq:DefSparseNorm} is fully equivalent to the more general topological definition of $\norm{f}{\ell^\tau(\dict)}$ introduced in~\cite{DeVore:1996aa}.}
\begin{equation}
\label{eq:DefSparseNorm}
\norm{f}{\ell^\tau(\dict)} := \inf_{c | \dict c = f} \norm{c}{\tau}
\end{equation}
and the associated {\em sparsity spaces} (also called {\em smoothness spaces}, for when $\dict$ is, e.g., a wavelet frame, they indeed characterize smoothness on the Besov scale) 
\[
\ell^\tau(\dict) := \{f, \norm{f}{\ell^\tau(\dict)}<\infty\}. 
\]
\subsection{Direct and inverse estimates}
Interestingly, the above defined concepts are related. In a Hilbert space $X = \cH$, when $\dict$ satisfies the upper bound
\begin{equation}
\label{eq:DefUpperFrameBound}
\norm{\dict c}{\cH}^2 \leq \ufbound \cdot \norm{c}{2}^2,\quad \forall c \in \ell^2,
\end{equation}
the sparsity spaces for $0<\tau<2$ are characterized as
\[
\ell^\tau(\dict) = \{f,\ \exists c,\ \norm{c}{\ell^\tau}<\infty, f = \dict c\}= \dict \ell^\tau,
\]
and for any $s>0$ we have the so-called {\em Jackson inequality}
\begin{equation}
\label{eq:DefJineq}
\sigma_k(f,\dict) \leq C_\tau(\ufbound) \cdot \norm{f}{\ell^\tau(\dict)} \cdot k^{-s},\quad s = \frac1\tau-\frac12,\quad \forall f\in \ell^\tau(\dict),\, \forall k\in\bN
\end{equation}
where, as indicated by the notation, the constant $C_\tau(\ufbound)$ only depends on $\tau$ and the upper bound $\ufbound$ in~\eqref{eq:DefUpperFrameBound}. 
Note that the upper bound~\eqref{eq:DefUpperFrameBound} holds true whenever the dictionary is a {\em frame}:  $\ufbound$ is then called the {\em upper frame bound}, and we will use this terminology.

When $\dict$ is an orthogonal basis, a converse result is true: if $\sigma_k(f,\dict)$ decays as $k^{-s}$ then $\norm{f}{\ell^\tau_w(\dict)} < \infty$, where $\ell^\tau_w$ is a weak $\ell^\tau$ space \cite{DeVore:1996aa} and $s = 1/\tau-1/2$. More generally, inverse estimates are related to a {\em Bernstein inequality} \cite{DeVore:1988aa,DeVore:1993aa}.

\begin{equation}
\label{eq:DefBineq}
\norm{f_k}{\ell^\tau(\dict)} \leq C \cdot  k^{r} \cdot \norm{f_k}{\cH},\quad \forall f_k \in \Sigma_k(\dict), \forall k.
\end{equation}
The  inequality \eqref{eq:DefBineq} is related to the so-called Bernstein-Nikolsky inequality, we refer the reader to \cite{Bang:1991aa,DeVore:1993aa} for more information.

\subsection{When approximation spaces are sparsity spaces}
When a Jackson inequality holds together with a Bernstein inequality with matching exponent $r = 1/\tau-1/2$, it is possible to characterize (with equivalent (quasi)-norms) the approximation spaces $\cA_q^s(\dict)$ as real interpolation spaces \cite[Chapter 7]{DeVore:1993aa} between $\cH$, denoted $(\cH,\ell^\tau(\dict))_{\theta,q}$, where $s = \theta r$, $0<\theta<1$.  
The definition of real interpolation spaces will be recalled in Section~\ref{sec:interpolation}. Let us just mention here that it is based on decay properties of the {\em $K$-functional}
\begin{equation}
\label{eq:DefKfunctional}
K(f,t;\cH,\ell^p(\dict)) 
= \inf_{c\in \ell^p} \left\{\norm{f-\dict c}{\cH} + t \norm{c}{p}\right\}.
\end{equation}
A priori, without a more explicit description of real interpolation spaces, the characterization of approximation spaces as interpolation spaces may seem just a sterile pedantic rewriting. 
Fortunately, we show in Section~\ref{sec:interpolation} ({\em Theorem~\ref{th:Interpolation}}) that the Bernstein inequality~\eqref{eq:DefBineq}, together with the upper frame bound~\eqref{eq:DefUpperFrameBound}, allows to directly identify approximation spaces  with sparsity spaces, with equivalent (quasi)-norms, for certain ranges of parameters. In particular, the following result can be obtained as a consequence of Theorem~\ref{th:Interpolation}.
\begin{theorem}
\label{th:FullSpaceEquality}
Suppose that $\dict$ satisfies the upper frame bound~\eqref{eq:DefUpperFrameBound} with constant $B$ as well as the Bernstein inequality~\eqref{eq:DefBineq} with some $0<\tau \leq 1$, with exponent $r=1/\tau-1/2$ and constant $C$. Then we have
\begin{equation}
\label{eq:FullSpaceEquality}
\cA_\tau^r(\dict) = \ell^\tau(\dict) 
\end{equation}
with equivalent norms, i.e. 
\[
c_1(B,C) \cdot \norm{f}{\ell^\tau(\dict)} 
\leq \norm{f}{\cA_\tau^r(\dict)} := \norm{f}{\cH}+\seminorm{f}{\cA_\tau^r(\dict)}
\leq c_2(B,C) \cdot \norm{f}{\ell^\tau(\dict)}
\] 
where the constants only depend on $B$ and $C$.
\end{theorem}
In other words, under the assumptions of Theorem~\ref{th:FullSpaceEquality}, a data vector $f \in \cH$ can be approximated by $k$-term expansions at rate $k^{-r}$ (in the sense $f \in \cA_\tau^r(\dict)$, where $r = 1/\tau-1/2$) if, {\em and only if}, $f$ admits a sparse {\em representation} $f = \sum_j c_j \phi_j$ with $\sum_j |c_j|^\tau < \infty$. 

\subsection{Ideal {\em vs} practical approximation algorithms}

Consider a function $f$ that can be approximated at rate $k^{-r}$ using $k$-term expansions from $\dict$: $\sigma_k(f,\dict) \lesssim k^{-r}, \forall k \geq 1$. Under the assumptions of the above Theorem, we can conclude that the function $f$ indeed admits a {\em representation} $f = \sum_j c_j \phi_j$ with $\sum_j |c_j|^\tau < \infty$. Suppose that we know how to compute such a representation (e.g., that we can solve the optimization problem $\min \|c\|_\tau\ \mbox{subject to}\ f = \dict c$). Then, sorting the coefficients in decreasing order of magnitude $|c_{j_m}| \geq |c_{j_{m+1}}|$, one can build a simple sequence of $k$-term approximants $f_m := \sum_{m=1}^k c_{j_m} \phi_{j_m}$ which converge to $f$ at the rate $r$: $\|f-f_k\|_\cH \lesssim k^{-r}$. Note that one may not however be able to guarantee that $\|f-f_k\|_\cH \leq C \sigma_k(f,\dict)$ for a fixed constant $C<\infty$. 

A special case of interest is $\tau=1$, where the optimization problem 
\[
\min \|c\|_1 \ \mbox{subject to}\ f = \dict c
\]
is convex, and the unit ball in $\ell^1(\dict)$ is simply the convex hull of the symmetrized dictionary $\{\pm \phi_j\}_j$ 
with $\phi_j$ the atoms of the dictionary $\dict$. Therefore, under the assumptions of the above Theorem for $\tau = 1$, if a function can be approximated at rate $k^{-1/2}$ then, after proper rescaling, it belongs to the convex hull of the symmetrized dictionary, and there exists constructive algorithms such as Orthonormal Matching Pursuit ~\cite{MZ93,PRK93} which are guaranteed to provide the rate of approximation $k^{-1/2}$ \cite[Theorem 3.7]{DeVore:1996aa}.

\subsection{Null Space Properties and fragility of Bernstein inequalities} 

On the one hand, it is known \cite{Gribonval:2004ab} that Jackson inequalities are always satisfied provided that the dictionary is a {\em frame}, i.e., 
\begin{equation}
\label{eq:DefFrameBounds}
\lfbound \norm{f}{\cH}^2 \leq \norm{\dict^T f}{2}^2 \leq \ufbound \norm{f}{\cH}^2,\quad \forall f \in \cH.
\end{equation}
The upper bound $B$ is actually equivalent to the upper frame bound~\eqref{eq:DefUpperFrameBound} and therefore sufficient for a Jackson inequality to hold. 

On the other hand, Bernstein inequalities are known to be much more subtle and seemingly fragile: they may be satisfied for certain structured dictionaries, but not for arbitrarily small perturbations thereof \cite{Gribonval:2004aa}. 

In Section~\ref{sec:fragile}, for the sake of simplicity we restrict our attention to the case $\tau=1$ when the dictionary $\dict$ forms a frame for a general Hilbert space $\cH$. We show that the Bernstein inequality for $\ell^1(\dict)$,
\begin{equation}\label{eq:Bern}
\|\dict c \|_{\ell^1(\dict)}\leq C m^{1/2} \|\dict c \|_{\cH},\qquad c :\norm{c }{0}\leq m,  \forall m\geq 1,
  \end{equation}
is closely linked to {\em properties of the kernel of $\dict$}  given by
\[
N(\dict):=\{z\in\ell^2:\dict z={0}\}.
\]

The seemingly simple case where we have a one dimensional null space for the dictionary, $N(\dict)=\Span\{z\}$ for some fixed sequence $z$, is particularly useful to demonstrate the fragility of the Bernstein estimates as the following example shows.
\begin{example}\label{example:1}
Given any infinite dictionary $\dict$ with 
$N(\dict)=\Span\{z\}$, where $z=(z_j)_{j=1}^\infty\in\ell^p$, for some $0<p\leq 1$. Then for each $\varepsilon>0$, there is a vector $\tilde{z}$ with $\norm{z-\tilde{z}}{p}<\varepsilon$ such that 
the Bernstein inequality \eqref{eq:Bern}  fails  for any dictionary $\tilde{\dict}$ with 
$N(\tilde{\dict})=\Span\{\tilde{z}\}$.

A specific case is given by $\dict=\cB\cup\{g\}$, with
$\cB$ the Dirac basis for $\ell^2$ and $g\in\ell^p$ for some $0<p<1$.  Then we can find an arbitrarily small perturbation $\tilde{g}$ of $g$ in $\ell^p$ such that 
the Bernstein inequality fails for the "perturbed" dictionary $\tilde{\dict}=\cB\cup\{\tilde{g}\}$.
\end{example}

Notice that in the preceeding example, nothing was asssumed about the 
Bernstein inequality for the dictionary $\dict$ 
itself. Thus, arbitrarily close to any dictionary with a reasonable one dimensional null space, there is a ''bad'' dictionary. 

However, it is possible to find good dictionaries with a one dimensional null space  for which  \eqref{eq:DefBineq} holds. The following example of such a dictionary.  

\begin{example}\label{example:2}
Suppose $\dict$ satisfies
$N(\dict)=\Span\{z\}$, where $z=(z_j)_{j=1}^\infty$  is such that  there is a constant $C<\infty$  satisfying  
\[
\forall k\in\bN: \sum_{j=k}^\infty |z_j|\leq C|z_k|.
\]
Then the Bernstein inequality \eqref{eq:Bern}  holds true. 

An explicit implementation of this example is given by $\dict=\cB\cup\{g\}$, with
$\cB=\{e_k\}_{k\in\bN}$ an orthonormal basis for $\ell^2$ and $g=-\sum_{k=1}^\infty a^k e_k$ for some fixed $0<a<1$.
\end{example}

Examples \ref{example:1} and \ref{example:2} combined show that one can always perturb a nice dictionary 
$\dict$ for which  \eqref{eq:DefBineq} holds ever so slightly as to make   \eqref{eq:DefBineq} collapse.

We justify the two examples in Section \ref{sec:fragile} by performing a careful analysis 
of the Bernstein inequality \eqref{eq:Bern} when $\dict$ is a frame. In Section \ref{sec:general_dict}  we study the general frame dictionary  and derive a sufficient condition  stated in Proposition \ref{prop:A}  for  \eqref{eq:Bern} to hold. Then in Section \ref{sec:1d_null_space} we present a more refined analysis (Proposition \ref{prop:B}) in the special case  where the kernel $N(\dict)$ is one-dimensional. 
The proof of Proposition \ref{prop:B} is based on an application of 
the general results in Section \ref{sec:general_dict}.

\subsection{Incoherence and the Restricted Isometry Property}
The above examples illustrate that the Bernstein inequality (and its nice consequences such as Theorem~\ref{th:FullSpaceEquality}) can be fairly fragile. However, this could be misleading, and we will now show that in a certain sense "most" dictionaries satisfy the inequality in a robust manner.

In a previous work we showed that {\em incoherent} frames \cite{Gribonval:2006aa} satisfy a "robust" Bernstein inequality, although with an exponent $r = 2(1/\tau-1/2)$ instead of the exponent $s = 1/\tau-1/2$ that would match the Jackson inequality. This inequality is then robust, because small enough perturbations of incoherent dictionaries remain incoherent.

In the last decade, a very intense activity related to Compressed Sensing \cite{Donoho:2006aa} has lead to the emergence of the concept of Restricted Isometry Property (RIP) \cite{Candes:2005aa,Candes:2006aa}, which generalizes the notion of coherence. A dictionary $\dict$ is said to satisfy the RIP of order $k$ with constant $\ripc$ if, for any coefficient sequence $c$ satisfying $\norm{c}{0}\leq k$, we have
\begin{equation}
\label{eq:DefRIP}
(1-\ripc) \cdot \norm{c}{2}^2 \leq \norm{\dict c}{\cH}^2 \leq (1+\ripc) \cdot \norm{c}{2}^2.
\end{equation}
The RIP has been widely studied for random dictionaries, and used to relate the minimum $\ell^1$ norm solution $c^\star$ of an inverse linear problem $f=\dict c$ to a "ground truth" solution $c_0$ which is assumed to be ideally sparse (or approximately sparse). 

In this paper, we are a priori not interested in "recovering" a coefficient vector $c_0$ from the observation $f = \dict c_0$. Instead, we wish to understand how the rate of ideal (but NP-hard) $k$-term approximation of $f$ using $\dict$ is related to the existence of a representation with small $\ell^\tau$ norm. 

In Section~\ref{sec:finitedim}, we study finite-dimensional dictionaries, where it turns out that the lower bound in the RIP~\eqref{eq:DefRIP} provides an appropriate tool to obtain Bernstein inequalities with controlled constant\footnote{The control of constants is the crucial part, since in finite dimension all norms are equivalent, which implies that the Bernstein inequality is always trivially satisfied.}. 
Namely, we say that the dictionary $\dict$ satisfies $\LRIP(k,\ripc)$ with a constant $\ripc<1$  provided that 
\begin{equation}\label{eq:RIP21}
 \norm{\dict c}{\cH}^2 \geq (1-\ripc) \cdot
  \norm{c}{2}^2,
\end{equation}
for any sequence $c$ satisfying $\norm{c}{0}\leq k$. 
We prove ({\em Lemma~\ref{le:RIPBineq}}) that in $\cH = \bR^N$ the lower frame bound $\lfbound>0$ and the $\LRIP(\slevel N,\ripc)$, imply a Bernstein inequality for $0<\tau \leq 2$ with exponent $r = 1/\tau-1/2$. As a result we have:
\begin{theorem}
\label{th:RIPimpliesCaract}
Let $\dict$ be an $m\times N$ frame with frame bounds $0<\lfbound \leq \ufbound < \infty$. Assume that $\dict$ satisfies  $\LRIP(\slevel N,\ripc)$, where $\ripc<1$ and $0<\slevel <1$. Then
\begin{itemize}
\item for $0<\tau\leq 2$, the Bernstein inequality~\eqref{eq:DefBineq} holds with exponent $r = 1/\tau-1/2$ and a constant $
C_\tau(\lfbound,\ripc,\slevel) < \infty$ ({\em cf} Eq.~\eqref{eq:ConstantBernstein})
\item for $0<\tau \leq 1$, $0<\theta<1$, we have, with equivalent norms,
\[
\cA_\tau^r(\dict) =  \ell^\tau(\dict) = (\cH,\ell^p(\dict))_{\theta,\tau}, \quad \frac1\tau = \frac{\theta}{p} + \frac{1-\theta}{2}.
\]
\end{itemize}
The constant $C_\tau(\lfbound,\ripc,\slevel)$ and the constants in norm equivalences may depend on $\lfbound, \ufbound, \ripc$, and $\slevel$, {\em but they do not depend on the dimension} $N$. 
\end{theorem}

For random Gaussian dictionaries, the typical order of magnitude of $\lfbound,\ripc(\slevel)$ is known and governed by the aspect ratio $\redund := N/m$ of the dictionary, provided that it is sufficiently high dimensional (its number of rows should be above a threshold $m(\redund)$  implicitly defined in Section~\ref{sec:finitedim}). We obtain the following theorem.
\begin{theorem}
\label{th:GaussianFullSpaceEquality}
Let $\dict$ be an $m \times N$ matrix with i.i.d. Gaussian entries $\mathcal{N}(0,1/m)$. Let $\redund:=N/m$ be the redundancy of the dictionary. If $m \geq m(\redund)$ then, except with probability at most $10\redund^2 \cdot \exp(-\gamma(\redund) m)$, we have 
for all $0<\tau\leq 1$ the equality
\begin{equation}
\label{eq:GaussianFullSpaceEquality}
\cA_\tau^r(\dict) = \ell^\tau(\dict) = (\cH,\ell^p(\dict))_{\theta,\tau},\quad r=1/\tau-1/2= \theta(1/p-1/2).
\end{equation}
with equivalent norms. 

The constants driving the equivalence of the norms are universal: they only depend on $\tau$ and the redundancy factor $\redund$ but {\em not} on the individual dimensions $m$ and $N$. Similarly $\gamma(\redund)$ and $m(\redund)$ only depend on $\redund$. 

For $\redund \geq 1.27$ we have $\gamma(\redund) > 7 \cdot 10^{-6}$. For large $\redund$ we have $\gamma(\redund)\approx 0.002$. 
\end{theorem}
Indeed, for random Gaussian dictionaries in high-dimension, with high probability, the Bernstein inequality holds for all $0< \tau \leq 2$ with constants driven by the aspect ratio $\redund := N/m$ but otherwise {\em independent of the dimension $N$}.Using the notion of decomposable dictionary \cite[Theorem 3.3]{Gribonval:2006aa}, this finite dimensional result can be easily adapted to build arbitrarily overcomplete dictionaries in infinite dimension that satisfy the equality~\eqref{eq:FullSpaceEquality}.


The result of Theorem~\ref{th:GaussianFullSpaceEquality} should be compared to our earlier result for incoherent frames obtained in   \cite{Gribonval:2006aa}. In  \cite{Gribonval:2006aa} we found an incoherent dictionary  with aspect ratio (approximately) 2 for which the Bernstein inequality \eqref{eq:DefBineq} can be shown to hold only for the  exponent $r=2(1/\tau-1/2)$, i.e., for $r$ twice as large as the Jackson exponent $s= 1/\tau-1/2$. 
Theorem~ \ref{th:GaussianFullSpaceEquality} illustrates that the result in  \cite{Gribonval:2006aa} really corresponds to a ``worst case'' behaviour and there are indeed many dictionaries (according to the Gaussian measure: the overwhelming majority of dictionaries) with a much better behaviour with respect to Bernstein estimates. This holds true even for aspect ratios $\redund$ that can be arbitrarily large.

\subsection{Conclusion and discussion}

The restricted isometry property is a concept that has been essentially motivated by the understanding of sparse regularization for linear inverse problems such as compressed sensing. 
Beyond this traditional use of the concept, we have shown new connections between the RIP and nonlinear approximation.

The main result we obtained is that, from the point of view of nonlinear approximation, a frame which satisfies a nontrivial restricted property $\delta_k < 1$ (i.e., in the regime $k \propto N$) behaves like an orthogonal basis: the optimal rate of $m$-term approximation can be achieved with an approach that does not involve solving a (potentially) NP-hard problem to compute the best $m$-term approximation for each $m$. In such nice dictionaries, near optimal $k$-term approximation can be achieved in two steps, like in an orthonormal basis: 
\begin{itemize}
\item decompose the data vector $f = \sum_j c_j \phi_j$, with coefficients as sparse as possible in the sense of minimum $\ell^\tau$ norm; 
\item keep the $m$ largest coefficients to build the approximant $f_m := \sum_{j \in I_m} c_j \phi_j$.
\end{itemize}
The second main result is that redundant dictionaries with the above property are not the exception, but rather the rule. While it is possible to build nasty overcomplete dictionaries either directly or by arbitrarily small perturbations of some "nice" dictionaries", in a certain sense the vast majority of overcomplete dictionaries are nice.

One should note that several results of this paper are expressed in finite dimension, where all norms are equivalent. The strength of the results is therefore not the mere existence of inequalities between norms, but in the fact that the involved constants do not depend on the dimension. From a numerical perspective, the control of these constants has essentially an impact in (very) large dimension, and it is not clear whether the constants numerically computed for random dictionaries are useful for dimensions less than a few millions.

A few key questions remains open. For a given data vector $f$, it is generally not known in advance to which $\ell^\tau(\dict)$ space $f$ belongs: 
under which conditions is it possible to efficiently compute a sparse decomposition $f = \sum_j c_j \phi_j$ which is guaranteed to be near optimal in the sense that $\|c\|_\tau$ is almost minimum whenever $f \in \ell^\tau(\dict)$? Can $\ell^1$ minimization (which is convex) be used and provide near best performance under certain conditions ? This is left to future work.

\section{Interpolation spaces}\label{sec:interpolation}
We recall the definition of the $K$-functional.
Let $Y$ be a (quasi-)normed space continuously embedded in a Hilbert space $\cH$. For $f\in\cH$, the $K$-functional is defined by

\[
K(f,t) = K(f,t;\cH,Y) 
:= \inf_{g\in Y} \left\{\norm{f-g}{\cH} + t \norm{g}{Y}\right\}
\]
and the norm defining the interpolation spaces $(\cH,Y)_{\theta,q}$, $0<\theta<1$, $0<q<\infty$, is given by:
\[
\norm{f}{(\cH,Y)_{\theta,q}}^q := \int_0^\infty [t^{-\theta} K(f,t)]^q \,\frac{dt}{t} \asymp \sum_{j \geq 0} 2^{j\theta q} K(f,2^{-j})^q.
\]
The interpolation space $(\cH,Y)_{\theta,q}$ is simply the set of $f$ for which the norm is finite. In our case we consider a frame dictionary $\dict$ and $Y=\ell^p(\dict)$, which is continuously embedded in $\cH$ for $0<p\leq 2$. We have the following result.
\begin{theorem}
\label{th:Interpolation}
Suppose $\dict$ is a frame dictionary for a Hilbert space $\cH$. Let $0<\tau\leq 1$ and suppose the Bernstein inequality for $\ell^\tau(\dict)$ holds with exponent $r$: 
\begin{equation*}
\label{eq:RecallBineq}
\norm{f_k}{\ell^\tau(\dict)} \leq C \cdot  k^{r} \cdot \norm{f_k}{\cH},\quad \forall f_k \in \Sigma_k(\dict), \forall k.
\end{equation*}
Define $\beta := r/(1/\tau-1/2)$. Then, for all $0<\theta<1$, we have the embedding
\begin{equation}\label{eq:interpolation1}
(\cH,\ell^p(\dict))_{\theta,\tau} \hookrightarrow \ell^\tau(\dict),\quad \frac1\tau = \frac{\theta/\beta}p + \frac{(1-\theta/\beta)}2;\\
  \end{equation}
Moreover, we have
\[
\cA_\tau^r(\dict) \hookrightarrow \ell^\tau(\dict),
\] 
and if in addition $r=1/\tau-1/2$ (i.e., $\beta=1$), then 
\[
\cA_\tau^r(\dict) = \ell^\tau(\dict)
\]
with equivalent norms. The constants in the norm inequalities depend only on $p$, on the Bernstein constant for $\ell^\tau(\dict)$, and on the upper frame bound for $\dict$.
\end{theorem}
\begin{proof}
We use the general technique proposed by DeVore and Popov \cite{DeVore:1988aa}, and adapt it to the special structure of the considered function spaces. 
One can check that with $Y = \ell^p(\dict)$ we have
\[
K(f,t) = \inf_c \left\{\norm{f-\dict c}{\cH} + t \norm{c}{p}\right\}.
\]
For each $j$ we consider $c_j$ an (almost) minimizer of the right hand side above for $t=2^{-j}$. Fix $0<\theta<1$ and define $s:=r/\theta$ and $p<2$ such that $s = 1/p-1/2$, and set $m_j = \lfloor 2^{j/s} \rfloor \asymp 2^{j/s}$. Define $\tilde{c}_j$ to match $c_j$ on its $m_j$ largest coordinates, and be zero anywhere else. Finally, define $f_0:=0$, $f_j := \dict \tilde{c}_j$, $j\in\bN$ . We can observe that
\begin{eqnarray*}
\norm{f-f_j}{\cH} 
& \leq &
\norm{f-\dict c_j}{\cH} + \norm{\dict (c_j-\tilde{c}_j)}{\cH}\\
& \stackrel{(a)}{\lesssim} &
\norm{f-\dict c_j}{\cH} + \norm{c_j-\tilde{c}_j}{2}\\
& \lesssim &
\norm{f-\dict c_j}{\cH} + \norm{c_j}{p} \cdot m_j^{-s}
\lesssim 
\norm{f-\dict c_j}{\cH} + \norm{c_j}{p} \cdot 2^{-j}\\
& \lesssim & K(f,2^{-j}),
\end{eqnarray*}
where in (a) we used the upper frame bound $\ufbound$ of $\dict$. Accordingly we get
\[
\norm{f_{j+1}-f_j}{\cH} \leq \norm{f-f_j}{\cH} + \norm{f-f_{j+1}}{\cH}
\leq C \cdot K(f,2^{-j})
\]
where the constant only depends on $p$ and the upper frame bound $\ufbound$ of $\dict$.
Since $\tau \leq 1$, we have the quasi-triangle inequality 
\[
\norm{u+v}{\ell^\tau(\dict)}^\tau \leq \norm{u}{\ell^\tau(\dict)}^\tau +\norm{v}{\ell^\tau(\dict)}^\tau.
\]
Since $f = \lim_{j \to \infty} f_j = \sum_{j=0}^\infty (f_{j+1}-f_j)$ we obtain
\begin{eqnarray*}
\norm{f}{\ell^\tau(\dict)}^\tau 
& \leq &
\sum_{j \geq 0} \norm{f_{j+1}-f_j}{\ell^\tau(\dict)}^\tau \\
& \stackrel{(b)}{\lesssim}&
\sum_{j \geq 0} \left[(2^{j/s})^r \norm{f_{j+1}-f_j}{\cH}\right]^\tau\\
& \lesssim&
\sum_{j \geq 0} \left[2^{j\theta} K(f,2^{-j}) \right]^\tau 
\asymp
\norm{f}{(\cH,\ell^p(\dict))_{\theta,\tau}}^\tau.
\end{eqnarray*}
In (b) we used the fact that $f_{j+1}-f_j \in \Sigma_m(\dict)$ with $m = m_j+m_{j+1} \lesssim 2^{j/s}$, 
and the assumption that the Bernstein inequality with exponent $r$ holds for $\ell^\tau(\dict)$.
To summarize we obtain $(\cH,\ell^p(\dict))_{\theta,\tau} \subset \ell^\tau(\dict)$, together with the norm inequality\[
\norm{f}{\ell^\tau(\dict)} \leq C \cdot \norm{f}{(\cH,\ell^p(\dict))_{\theta,\tau}}
\]
where the constant only depends on the Bernstein constant for $\ell^\tau(\dict)$, on $p$, and on the upper frame bound $\ufbound$ for $\dict$. 
We have
$1/\tau - 1/2 = r/\beta = (r/\beta s) s  = (\theta/\beta) (1/p-1/2)$, i.e., 
$1/\tau = (\theta/\beta)/p + (1-\theta/\beta)/2$.

Similarly, we can define $f_0=0$ and $f_j$ a (near)best $m_j$-term approximation to $f$ with $m_j = 2^{j-1}$, $j \geq 1$ and obtain $\|f_{j+1}-f_j\|_\cH \leq 2 \sigma_{2^{j-1}}(f,\dict), j \geq 1$. Using the Bernstein inequality and derivations essentially identical to the previous lines we get
\[
\|f\|_{\ell^\tau(\dict)}^\tau 
\lesssim \|f_1\|_\cH^\tau + \sum_{j\geq 1} \left[2^{(j-1)r} \sigma_{2^{j-1}}(f,\dict) \right]^\tau \lesssim \|f\|_{\cA^r_\tau(\dict)}^\tau.
\]
The constant only depends on the Bernstein constant for $\ell^\tau(\dict)$.

Using~\cite[Theorem~6]{Gribonval:2004ab}, the upper frame bound implies 
the continuous embedding $\ell^\tau(\dict) \hookrightarrow \cA_\tau^s(\dict)$ with $s = 1/\tau-1/2$. Hence, when the Berstein exponent is $r=1/\tau-1/2=s$  we have equality that is to say $\cA_\tau^r(\dict) = \ell^\tau(\dict)$ with equivalent norms.
\end{proof}

\begin{remark}
A consequence of Theorem~\ref{th:Interpolation} is a partial answer to an open question raised in~\cite{Gribonval:2006aa}, where "blockwise incoherent dictionaries" are considered and a Bernstein inequality is proved, with exponent $r = \beta(1/\tau-1/2)$, $\beta=2$, for all $0<\tau \leq 2$, yielding the two-sided embedding \cite[Theorem 3.2]{Gribonval:2006aa}:
\[
\ell^\tau(\dict) \hookrightarrow \cA_q^s(\dict) \hookrightarrow (\cH,\ell^\tau(\dict))_{1/2,q},\quad 0 < \tau \leq 2,\quad \tau \leq q < \infty,\quad s=1/\tau-1/2.
\] 
By Theorem~\ref{th:Interpolation}, for $0<q \leq 1$, the Bernstein inequality with exponent $r = \beta(1/q-1/2)$ further implies the embedding $(\cH,\ell^\tau(\dict))_{1/2,q} \hookrightarrow \ell^q(\dict)$ where $1/q = 1/(4\tau) + 3/8$, i.e., $q = 8\tau/(3\tau+2)$. As a result we have
\[
\ell^\tau(\dict) \hookrightarrow \cA_q^s(\dict) \hookrightarrow \ell^q(\dict),
\quad 0<\tau \leq 2/5, \quad q = 8\tau/(3\tau+2), \quad s = 1/\tau-1/2.
\]
We know from~\cite{Gribonval:2006aa} an example of blockwise incoherent dictionary where the exponent of the Berstein inequality cannot be improved, hence the above embedding is also sharp for this class of dictionaries.

\end{remark}
\section{Bernstein estimates for frame dictionaries}
\label{sec:fragile}

In this section we are interested in the Bernstein inequality \eqref{eq:Bern}
in the general case where the dictionary $\dict$ forms a frame for a Hilbert space $\cH$. 
The dimension of $\cH$ may be finite or infinite. 
We will show that the Bernstein inequality is closely linked to properties of the kernel of $\dict$  given by
$$N=N(\dict):=\{z\in\ell^2:\dict z={0}\}.$$
In fact, the frame property ensures that $\norm{\dict c }{\cH}\asymp \inf_{z\in N}\norm{c +z}{2}$ for any sequence $c \in \ell^2$. 
Hence, the Bernstein inequality \eqref{eq:Bern}  holds if and only if the quantity  
 \begin{equation}\label{eq:Cdict}
C(\dict):=
\sup_{m\in\bN}\sup_{c :\norm{c }{0}\leq m}\sup_{z\in N}\inf_{v\in N}m^{-1/2} \cdot \frac{\norm{c +v}{1}}{\norm{c +z}{2}} 
  \end{equation}
is finite. 

We split our analysis in two parts. In Section \ref{sec:general_dict} we 
derive an upper bound on $C(\dict)$ that results in 
 a sufficient condition for \eqref{eq:Bern} to 
hold for a general frame dictionary (Proposition \ref{prop:A}). 
In Section \ref{sec:1d_null_space} we specialize to the case where the kernel $N(\dict)$ is one-dimensional.

The analysis in Section \ref{sec:1d_null_space} is  used to  justify
Examples \ref{example:1} and \ref{example:2}.

\subsection{Bernstein constant for general dictionaries}\label{sec:general_dict}
Here we derive an upper estimate of the quantity  $C(\dict)$, given by \eqref{eq:Cdict}, for general frame dictionaries in a Hilbert space. This estimate leads to the following sufficient condition for a Bernstein inequality for such  dictionaries.
\begin{proposition}
\label{prop:A}
Suppose the dictionary $\dict$ forms a frame for the Hilbert space $\cH$, and $\dict$ has  kernel $N:=N(\dict)$. Then 
the Bernstein inequality \eqref{eq:Bern} holds provided that
\begin{equation}\label{eq:B_upper}
  \sup_{z\in N}\sup_{m\in\bN}\sup_{I:|I|\leq m} m^{-1/2} \cdot \frac{\norm{z_{I^c}}{1}}{\norm{z_{I^c}}{2}}<\infty.
  \end{equation}
Moreover, in the  case where $N\cap \ell^1=\{{0}\}$, the Bernstein inequality \eqref{eq:Bern}  holds if and only if
\begin{equation}\label{eq:C1}
  C_1(\dict):=\sup_{z\in N}\sup_{m\in \bN}\frac{\norm{z_m}{1}}{m^{1/2}\sigma_m(z)_2}<\infty,
  \end{equation}
  where $z_m$ is the vector containing the $m$ largest entries in $z$.   
\end{proposition}
\begin{proof}
We prove the sufficient condition for \eqref{eq:Bern} by 
deriving an upper bound for
$C(\dict)$ given by \eqref{eq:Cdict}. For any $m\in\bN$, 
\begin{align}
\sup_{c :\norm{c }{0}\leq m}\sup_{z\in N}\inf_{v\in N}\frac{\norm{c +v}{1}}{m^{1/2}\norm{c +z}{2}}&=
\sup_{I:|I|\leq m}
\sup_{c :\supp(c )\subseteq I}
\sup_{z\in N}\inf_{v\in N}\frac{\norm{c +v}{1}}{m^{1/2}\norm{c +z}{2}}\nonumber\\
&=
\sup_{I:|I|\leq m}
\sup_{c :\supp(c )\subseteq I}
\sup_{z\in N}\inf_{v\in N}\frac{\norm{c +v_I}{1}+\norm{v_{I^c}}{1}}{m^{1/2}\sqrt{\norm{c +z_I}{2}^2+\norm{z_{I^c}}{2}^2}}\nonumber\\
&\leq
\sup_{I:|I|\leq m}
\sup_{c :\supp(c )\subseteq I}
\sup_{z\in N}\frac{\norm{c +z_I}{1}+\norm{z_{I^c}}{1}}{m^{1/2}\sqrt{\norm{c +z_I}{2}^2+\norm{z_{I^c}}{2}^2}}.\label{eq:Cest}
\end{align}
For a given support $I$ and $z\in N$, we introduce $\gamma_I^z:=\norm{z_{I^c}}{1}/\norm{z_{I^c}}{2}$. Notice that for $|I|\leq m$ and $c $ with $\supp(c )\subseteq I$, we have
\begin{align}
\norm{c +z_I}{1}+\norm{z_{I^c}}{1}&\leq 
m^{1/2}\norm{c +z_I}{2}
+\gamma_I^z\norm{z_{I^c}}{2}\nonumber\\
&\leq
\max\{m^{1/2},\gamma_I^z\}\big(\norm{c +z_I}{2}+\norm{z_{I^c}}{2}\big)\nonumber\\
&\leq \sqrt{2}\max\{m^{1/2},\gamma_I^z\}\sqrt{\norm{c +z_I}{2}^2+\norm{z_{I^c}}{2}^2}.\label{eq:gamma}
\end{align}
We let $\gamma_m^z:=\sup_{I:|I|\leq m}\gamma_I^z$. Hence, from \eqref{eq:Cest} we deduce that 
$$C(\dict)\leq \sqrt{2}\max\bigg\{1,\sup_{z\in N}\sup_{m\in\bN}\frac{\gamma_m^z}{m^{1/2}}\bigg\},$$
which shows that condition \eqref{eq:B_upper} implies  $C(\dict)<\infty$.

Let us now consider the case $N\cap \ell^1=\{{0}\}$. Notice that the infimum over $v\in N$ in \eqref{eq:Cdict} is attained for $v=0$.
Hence, $C(\dict)=\sup_{z\in N} B_z$, with  
\begin{equation}\label{eq:B0}
    B_z:=\sup_{m\in\bN}\sup_{I:|I|\leq m}
\sup_{c :\supp(c )\subseteq I}\frac{\norm{c }{1}}{m^{1/2}\sqrt{\norm{c + z_I}{2}^2+\norm{z_{I^c}}{2}^2}}.
\end{equation}
For a fixed support $I$,  standard  estimates show that
$$\frac{\norm{c }{1}}{m^{1/2}\sqrt{\norm{c + z_I}{2}^2+\norm{z_{I^c}}{2}^2}}$$
is maximal for choices of the type
$c =-[z_I+\lambda\text{sign}(z)\mathbf{1}_I$], $\lambda>0$. This choice of $c $ leads to the corresponding (squared) optimization problem
\begin{equation*}\label{eq:opti}
\sup_{\lambda\in \bR} \frac{(\lambda m+\norm{z_I}{1})^2}{m(\lambda^2m^2+\norm{z_{I^c}}{2}^2\big)}=\frac1m\bigg(\frac{\norm{z_I}{1}^2}{\norm{z_{I^c}}{2}^2}+1\bigg).
    \end{equation*}
 Notice that
 \[
 \sup_{I:|I|\leq m}\frac{\norm{z_I}{1}}{\norm{z_{I^c}}{2}}=\frac{\norm{z_m}{1}}{\sigma_m(z)_2},
 \]
 so we deduce that
  \begin{equation}\label{eq:C1_B0}
   C_1(\dict)\leq C(\dict)\leq C_1(\dict)+1,  
  \end{equation}
  which completes the proof.
\end{proof}

\subsection{Dictionaries with one dimensional null-spaces}\label{sec:1d_null_space}
We now turn to the simplified case where the dictionary $\dict$ has a one-dimensional null-space.
In this case, we derive necessary conditions for the Bernstein inequality \eqref{eq:Bern}  to hold that is valid even when $N(\dict)\subset \ell^1$, a case not covered by the  necessary condition of Proposition \ref{prop:A}.
We prove the following:  

\begin{proposition}\label{prop:B}
Suppose the dictionary $\dict$ is a frame for the Hilbert space $\cH$ and has a one-dimensional null-space, 
 $N(\dict)=\Span\{z\}$. Also suppose 
  the Bernstein inequality \eqref{eq:Bern}  holds. Then 
  \begin{equation}\label{eq:Bl}
C_2(z):=
\sup_{m\in\bN}\sup_{I:|I|\leq m}\min\bigg(
\frac{\norm{z_I}{1}}{m^{1/2}\norm{z_{I^c}}{2}},\frac{\norm{z_{I^c}}{1}}{m^{1/2}\norm{z_{I^c}}{2}}\bigg)<\infty.
  \end{equation}
Moreover, if $z\in\ell^p$ for some $0<p<1$, and the Bernstein inequality \eqref{eq:Bern}  holds for $\dict$, then
\[\sup_{m\in\bN} \frac{\sigma_m(z)_1}{m^{1/2}\sigma_m(z)_2}<\infty. \]
\end{proposition}

\begin{proof}[Proof of Proposition \ref{prop:B}]
In this setting, the Bernstein inequality \eqref{eq:Bern}
 holds if and only if the quantity  
\[
C(\dict):=\sup_{m\in\bN}\sup_{c :\norm{c }{0}\leq m}\sup_{\mu\in\bR}\inf_{\lambda\in\bR}\frac{\norm{c +\lambda z}{1}}{m^{1/2}\norm{c +\mu z}{2}}
\]
is finite. By rescaling, we have 
\begin{align}
C(\dict)
&=\sup_{m\in\bN}\sup_{c :\norm{c }{0}\leq m}\sup_{\mu\in\bR}\inf_{\lambda\in\bR} \frac{\norm{\frac{\lambda}{\mu}z+\frac{1}{\mu}c }{1}}{m^{1/2}\norm{\frac{1}{\mu}c + z}{2}}\nonumber\\
&=\sup_{m\in\bN}\sup_{\tilde{c }:\norm{{\tilde c }}{0}\leq m}\inf_{\delta\in\bR} \frac{\norm{\delta z+\tilde{c }}{1}}{m^{1/2}\norm{\tilde{c }+ z}{2}}\nonumber\\
&=
\sup_{m\in\bN}\sup_{I:|I|\leq m}
\sup_{c :\supp(c )\subseteq I}\inf_{\delta\in\bR} \frac{\norm{\delta z+c }{1}}{m^{1/2}\norm{c + z}{2}}\nonumber\\
&=
\sup_{m\in\bN}\sup_{I:|I|\leq m}
\sup_{c :\supp(c )\subseteq I}\inf_{\delta\in\bR} \frac{\norm{\delta z_I+c }{1}+\norm{\delta z_{I^c}}{1}}{m^{1/2}\sqrt{\norm{c + z_I}{2}^2+\norm{z_{I^c}}{2}^2}}.\label{eq:CPhi}
\end{align}
To get a lower estimate for $C(\dict)$, we simply chose $c =-z_I$ 
 in \eqref{eq:CPhi}
to obtain
\begin{align}
C(\dict)&\geq \sup_{m\in\bN}\sup_{I:|I|\leq m} \inf_{\delta\in\bR}
\frac{|\delta-1|\norm{z_I}{1}+|\delta|\norm{z_{I^c}}{1}}{m^{1/2}\norm{z_{I^c}}{2}}\nonumber\\
&= \sup_{m\in\bN}\sup_{I:|I|\leq m}\min\bigg(
\frac{\norm{z_I}{1}}{m^{1/2}\norm{z_{I^c}}{2}},\frac{\norm{z_{I^c}}{1}}{m^{1/2}\norm{z_{I^c}}{2}}\bigg)\nonumber\\
&\geq \sup_{m\in\bN}\min\bigg(
\frac{\norm{z_m}{1}}{m^{1/2}\sigma_m(z)_2},
\frac{\sigma_m(z)_1}{m^{1/2}\sigma_m(z)_2}\bigg).\label{eq:Bern_lower}
\end{align}
Then clearly $C(\dict)<\infty$ implies condition \eqref{eq:Bl}.

If, in addition, we have $z\in\ell^p$ for some $0<p<1$, then it follows from standard results on nonlinear approximation with bases in $\ell^2$, see \cite{DeVore:1996aa}, that $m^{1/2}\sigma_m(z)_2\rightarrow 0$ as $m\rightarrow \infty$. Thus 
\[\frac{\norm{z_m}{1}}{m^{1/2}\sigma_m(z)_2}\rightarrow \infty,\]
and we conclude from \eqref{eq:Bern_lower}  that
\[\sup_{m\in\bN} \frac{\sigma_m(z)_1}{m^{1/2}\sigma_m(z)_2}<\infty.\]
\end{proof}

We now turn to a justification of Examples \ref{example:1} and \ref{example:2} using Propositions \ref{prop:A} and \ref{prop:B}.

\subsubsection{Examples \ref{example:1} and \ref{example:2} revisited }
We first verify the claim made in Example~\ref{example:1}.
 Given any dictionary $\dict$ with 
$N(\dict)=\Span\{z\}$, where $z=(z_j)_{j=1}^\infty\in\ell^p$, for some $0<p\leq 1$.
For any $\epsilon>0$, we modify $z$ as follows
\begin{itemize}
\item[1.] Choose $m_0\geq 2$ such that $\sum_{j=m_0}^\infty |z_j|^p<\epsilon/2$.
\item[2.] Choose a sequence $\{m_\ell\}_{\ell=1}^\infty$ satisfying $m_{\ell+1}/m_\ell\rightarrow \infty$. Notice that the sequence will necessarily have super-exponential growth.
\item[3.] Fix $\beta>1/p\geq 1$, and choose $\{\gamma_j\}_{j=0}^\infty$ such that $\gamma_\ell:=C(m_{\ell+1}-m_\ell)^{-\beta}$, with the constant $C$  defined by the equation $\sum_{\ell=0}^\infty \gamma_\ell^p[m_{\ell+1}-m_\ell]=\sum_{j=m_0+1}^\infty |z_j|^p$.
\item[4.] Now define $\tilde{z}=(\tilde{z}_j)_{j=0}^{\infty}$ by
$$\tilde{z}_j:=\begin{cases}
z_j,&0\leq j\leq m_0\\
\gamma_\ell,& j\in [m_{\ell}+1,m_{\ell+1}],\quad \ell\in \bN_0.
\end{cases}$$
\end{itemize}
It is easy to verify (using 1. and 3.) that $\norm{z-\tilde{z}}{p}^p<\epsilon$. 

Let us consider the index set $I=[1,m_\ell]$, $\ell\geq 1$. We have,
\begin{align*}
 \frac{\norm{\tilde{z}_{I^c}}{1}^2}{\norm{\tilde{z}_{I^c}}{2}^2}&\geq \frac{[(m_{\ell+1}-m_\ell)\gamma_\ell]^2}{C^2\sum_{k=\ell}^\infty (m_{k+1}-m_k)^{1-2\beta}}\\
&\geq \frac{(m_{\ell+1}-m_\ell)^{2-2\beta}}{(m_{\ell+1}-m_\ell)^{1-2\beta}} \\
&\geq m_{\ell+1}-m_\ell.
\end{align*}
Thus,
$$
\frac{\norm{\tilde{z}_{I^c}}{1}^2}{m_\ell\norm{\tilde{z}_{I^c}}{2}^2}
\geq \frac{m_{\ell+1}-m_\ell}{m_\ell}=
\frac{m_{\ell+1}}{m_\ell}-1
\rightarrow \infty,$$
as $\ell\rightarrow \infty$. Also, since $\beta>1$,
\begin{align*}
 \frac{\norm{\tilde{z}_{I}}{1}^2}{m_\ell\norm{\tilde{z}_{I^c}}{2}^2}&\geq \frac{C'}{C^2\sum_{k=\ell}^\infty (m_{k+1}-m_k)^{1-2\beta}}\\
&\geq \frac{C'}{m_\ell(m_{k+1}-m_k)^{1-2\beta}} \\
&\geq \frac{m_{\ell+1}-m_\ell}{m_\ell}\rightarrow \infty,
\end{align*}
as $\ell\rightarrow \infty$. We conclude that $C_2(\tilde{z})=\infty$,
with $C_2(\tilde{z})$ given in Proposition \ref{prop:B}.

To verify the claim made in Example \ref{example:2}, suppose $\dict$ satisfies
$N(\dict)=\Span\{z\}$, where $z=(z_j)_{j=1}^\infty$  is such that  there is a constant $C<\infty$  satisfying  
\[
\forall k\in\bN: \sum_{j=k}^\infty |z_j|\leq C|z_k|.
\]
Then for any finite index set $I\subset \bN$, we let $m_I:=\min\{k:k\not \in I\}$, and notice that
$\norm{z_{I^c}}{1}\leq \sum_{j=m_I}^\infty |z_j|\leq  C|z_{m_I}|$, while 
$\norm{z_{I^c}}{2}\geq |z_{m_I}|$. Hence,
$$\frac{\norm{z_{I^c}}{1}}{\norm{z_{I^c}}{2}}\leq C\frac{|z_{m_I}|}{|z_{m_I}|}=C,$$
and
\[
\sup_{m\in \bN}\sup_{I:|I|\leq m} \frac{\norm{z_{I^c}}{1}}{m^{1/2}\norm{z_{I^c}}{2}}\leq C<\infty,
\]
so \eqref{eq:B_upper} is satisfied and the Bernstein inequality \eqref{eq:Bern}  holds by Proposition \ref{prop:A}.

\section{Bernstein inequality and the RIP}
\label{sec:finitedim}

For certain incoherent dictionaries studied in \cite{Gribonval:2006aa}, the Berstein inequality {\em cannot} match the Jackson inequality, but it still holds with a sharp exponent $r = 2(1/\tau-1/2)$ for any $\tau \leq 2$, i.e. the sharp factor that can be used in Theorem~\ref{th:Interpolation} is $\beta=2$. This result exploits incoherence \cite[Lemma 2.3]{Gribonval:2006aa} to prove that the lower bound in the RIP is satisfied for $k$ of the order of $\sqrt{N}$. 
Below we prove that the lower frame bound~\eqref{eq:DefFrameBounds}, together with the lower bound in the RIP~\eqref{eq:RIP21} with $k$ {\em of the order of $N$}, implies the Bernstein inequality~\eqref{eq:DefBineq} with controlled constant and exponent matching that of the Jackson inequality~\eqref{eq:DefJineq}. This Lemma therefore extends our previous result based on incoherence \cite[Theorem 2.1]{Gribonval:2006aa}.
\begin{lemma}
\label{le:RIPBineq}
Let $\dict$ be an $m\times N$ dictionary. Suppose $\dict$ has lower frame bound $\lfbound>0$  and satisfies  $\LRIP(\slevel N,\ripc)$, where $\ripc<1$ and $0<\slevel <1$. Then for $0<\tau\leq 2$, the Bernstein inequality~\eqref{eq:DefBineq} holds with exponent $r = 1/\tau-1/2$ and constant
\begin{equation}\label{eq:ConstantBernstein}
C_\tau(\lfbound,\ripc,\slevel):=\max\big\{(1-\ripc)^{-1/2},\lfbound^{-1/2}\slevel^{1/2-1/\tau}\}.
  \end{equation}
 \end{lemma}
\begin{proof}
First, suppose $1\leq k\leq \slevel N$. Take $f\in\Sigma_k(\dict)$, and write
$f=\dict c$ with $\norm{c}{0}\leq k$. Then, by the  $\LRIP(\slevel N,\ripc)$ condition,
\[
\norm{f}{\ell^\tau(\dict)}\leq
\norm{c}{\tau}\leq 
k^{1/\tau-1/2}\norm{c}{2}\leq 
(1-\ripc)^{-1/2}k^{1/\tau-1/2} \cdot \norm{\dict c}{\cH}.
\]
For $\slevel N\leq k\leq N$, take $f\in\Sigma_k(\dict)$. We express $f$ in terms of its canonical frame expansion relative to $\dict$,
\begin{equation}\label{eq:frame}
  f=\sum_{j=1}^N \langle f,\tilde{\phi}_j\rangle \phi_j.
  \end{equation}
We recall that the dual frame $\{\tilde{\phi}_j\}$ has an upper frame bound $\lfbound^{-1}$. Hence, we can use the expansion \eqref{eq:frame} to deduce that
\begin{align*}
  \norm{f}{\ell^\tau(\dict)}&\leq
\norm{\{\langle f,\tilde{\phi}_j\rangle\}}{\tau}
\leq N^{1/\tau-1/2}
\norm{\{\langle f,\tilde{\phi}_j\rangle\}}{2}\\&
\leq 
\lfbound^{-1/2}N^{1/\tau-1/2} \cdot \norm{\dict c}{\cH}\\
&\leq [\lfbound^{-1/2}\slevel^{1/2-1/\tau}] k^{1/\tau-1/2} \cdot \norm{\dict c}{\cH}.
\end{align*}
The Bernstein inequality and its constant now follow at once from the two separate estimates.
\end{proof}

Lemma~\ref{le:RIPBineq} proves half of Theorem~\ref{th:RIPimpliesCaract}. Let us complete the proof now.
\begin{proof}[Proof of Theorem~\ref{th:RIPimpliesCaract}]
As we have seen, the lower frame bound and the $\LRIP(\slevel N,\ripc)$ property imply the Bernstein inequality for all $0<\tau \leq 2$. Moreover the upper frame bound implies a Jackson inequality. For $0 < p < \tau \leq1$ both the Jackson and Bernstein inequalities hold for $\ell^p(\dict)$ with exponent $1/p-1/2$, hence \cite[Chapter 7]{DeVore:1993aa} we have with equivalent norms
\[
\cA_\tau^r(\dict) = (\cH,\ell^p(\dict))_{\theta,\tau},\quad r = \theta (1/p-1/2),\quad 0<\theta<1.
\]
The Bernstein inequality also holds for $\ell^\tau(\dict)$ with exponent $r = 1/\tau-1/2$, hence by Theorem~\ref{th:FullSpaceEquality} $\cA_\tau^r(\dict) = \ell^\tau(\dict)$ with equivalent norms.
\end{proof}

Next we wish to estimate $\lfbound,\ufbound,\ripc,\slevel$ when $\dict$ is a random Gaussian dictionary.
The following Lemma summarizes well known facts (see e.g.~\cite{Candes:2006aa,Richard-Baraniuk:2008aa}).
\begin{lemma}
\label{le:RIPRandom}
Let $\dict$ be an $m \times N$ matrix with i.i.d. Gaussian entries $\mathcal{N}(0,1/m)$.
For any $\epsilon > 0$ and $1 \leq k < m$, it satisfies the $\LRIP(k,\ripc)$ with $1-\ripc = (1-\eta)^2$, where
\begin{equation}
\label{eq:DefLRIPConstRandom}
\eta := 
\sqrt{\frac km} \cdot 
\left(1+ (1+\epsilon) \cdot 
\sqrt{2 \cdot \left(1+ \log \frac{N}{k}\right) }\right)
\end{equation}
except with probability at most 
\begin{equation}
\exp\left( -2 \epsilon k \cdot (1+ \log \frac{N}{k})\right).
\end{equation}
Moreover, except with probability at most $\exp(-\epsilon^2 m/2)$, it has the lower frame bound
\begin{align}
\lfbound & \geq  (\sqrt{N/m}-1-\epsilon)^2\\
\intertext{and, except with probability at most $\exp(-\epsilon^2 m/2)$, it has the upper frame bound}
\ufbound & \leq   (\sqrt{N/m}+1+\epsilon)^2
\end{align}
\end{lemma}

\begin{proof}
First, for a given index set $\Lambda$ of cardinality $k < m$, we observe that the restricted matrix $\dict_\Lambda$ is $m \times k$ with i.i.d. Gaussian entries $\mathcal{N}(0,1/m)$, hence its smallest singular value exceeds $1-\sqrt{k/m}-t$ except with probability at most $\exp(- m t^2/2)$ \cite[Theorem II.13]{K.-R.-Davidson:2001aa}. By a union bound, the smallest singular values among all submatrices $\dict_\Lambda$ associated to the $\binom{N}{k}$ possible index sets $\Lambda$ of cardinality $k$ exceeds $1-\sqrt{k/m}-t$, except with probability at most
\(
p(t) := \binom{N}{k} \cdot \exp(- m t^2/2).
\)
Since for all $N,k$ we have $\binom{N}{k} \leq (Ne/k)^k = \exp\left( k \cdot (1+ \log \frac{N}{k}) \right)$, it follows that
\[
p(t) \leq \exp\left( k \cdot (1+ \log \frac{N}{k}) - m t^2/2\right).
\]
For $\epsilon>0$ we set
\[
t := (1+\epsilon) \cdot \sqrt{\textstyle \frac{2k}{m} \cdot \left(1+ \log \frac{N}{k}\right) }
\]
and obtain that, except with probability at most 
\begin{eqnarray*}
p(\epsilon)
&\leq&
\exp\left( k \cdot (1+ \log \frac{N}{k}) \cdot \left(1- (1+\epsilon)^2 \right)\right) 
\leq
\exp\left( -2\epsilon k \cdot (1+ \log \frac{N}{k})  \right)
\end{eqnarray*}
we have: for all $k$-sparse vector $c$ with $\norm{c}{0} = k$,
\begin{eqnarray*}
\norm{\dict c}{2} 
&\geq&
\left[
1-\sqrt{k/m} \cdot \left(1+ (1+\epsilon) \cdot 
\sqrt{2 \cdot \left(1+ \log \frac{N}{k}\right) }
\right)
\right] \cdot \norm{c}{2}.
\end{eqnarray*}
To control the frame bounds we consider the random matrix $\mathbf{\Psi} := \sqrt{\frac{m}{N}} \dict^T$. Since $\mathbf{\Psi}$ is $N \times m$ with i.i.d. Gaussian entries $\mathcal{N}(0,1/N)$, for any $t>0$ all its singular values exceed $1-\sqrt{m/N}-t$, except with probability at most $\exp(- N t^2/2)$ \cite[Theorem II.13]{K.-R.-Davidson:2001aa}. Setting $t = \epsilon \cdot \sqrt{m/N}$, since
\(
\|\dict^Tx\|_2^2 = \frac{N}{m} \|\mathbf{\Psi} x\|_2^2,
\)
we obtain that $\dict$ has lower frame bound
\[
\sqrt \lfbound  \geq  \sqrt{\frac{N}{m}} \cdot \left(1-(1+\epsilon) \cdot \sqrt{m/N}\right) = \sqrt{N/m}-1-\epsilon
\]
except with probability at most $\exp(-\epsilon^2 m/2)$. We proceed identically for the upper frame bound, using the fact that for any $t>0$, no singular value of $\mathbf{\Psi}$ exceeds $1+\sqrt{m/N}+t$, except with probability at most $\exp(- N t^2/2)$ \cite[Theorem II.13]{K.-R.-Davidson:2001aa}.
\end{proof}

We now obtain our first main theorem (Theorem~\ref{th:FullSpaceEquality}) by controlling the constant $\ripc$ from below when $k/m$ is bounded from above, given the redundancy $\redund = N/m$ of the dictionary $\dict$. 
\begin{proof}[Proof of Theorem~\ref{th:FullSpaceEquality}]
In Appendix~\ref{app:IncreasingFunction} we exhibit a threshold $t(\redund) \in (0,1)$ such that if  $N/m = \redund$ and $t= k/m \leq t(\redund)$ then
\[
\eta := \sqrt{\frac km} \cdot 
\left(1+ 2 \cdot 
\sqrt{2 \cdot \left(1+ \log \frac{N}{k}\right) }\right)
\leq 1/2.
\]
Consider $k := \lfloor t(\redund) m \rfloor$. By Lemma~\ref{le:RIPRandom} the dictionary $\dict$ satisfies the $\LRIP(k,\ripc)$ with $(1-\ripc)^{-1/2} = (1-\eta)^{-1} =2$ except with probability at most
\begin{eqnarray*}
p_1 = \exp\left( -2 k \cdot \left(1+ \log \frac{N}{k}\right)\right) 
&\leq&
\exp\left( -2 \left[t(\redund) m -1\right]  \cdot (1+ \log \redund)\right)\\
& \leq&
e^{2(1+\log \redund)} \cdot \exp\big( -2 t(\redund) (1+ \log \redund) m\big)
\end{eqnarray*}
Moreover, setting $\epsilon := (\sqrt\redund-1)/2$, it has lower frame bound $\sqrt\lfbound \geq \sqrt{N/m}-1-\epsilon \geq (\sqrt\redund-1)/2$ except with probability at most $p_2 = \exp(-\epsilon^2 m/2)$. 
For $m \geq m(\redund) := 2/t(\redund)$ we have 
\[
\slevel = \frac kN = \frac km \cdot \frac mN \geq \frac{t(\redund)-\frac1m}{\redund} \geq \frac{t(\redund)}{2\redund}.
\]
and by Lemma~\ref{le:RIPBineq}, we obtain (except with probability at most $p_1 + p_2$) that the Bernstein inequality holds for each $0<\tau \leq 2$ with constant
\begin{equation}
\label{eq:DefBineqRandomConstant}
C_\tau(\redund) \leq \max\left(2,2(\sqrt{\redund}-1)^{-1} \cdot \left[\frac{t(\redund)}{2\redund}\right]^{1/2-1/\tau}\right).
\end{equation}
Since we also have the upper frame bound $\sqrt{\ufbound} \leq \sqrt \redund + 1 + \epsilon'$ except with probability at most $p_3 = \exp(-(\epsilon')^2 m/2)$ we obtain with $\epsilon'=1$ that the upper frame bound $\sqrt{\redund}+2$ together with the Bernstein inequality with constant $C_\tau(\redund)$ jointly hold, except with probability at most $p_1+p_2 + p_3 \leq \beta \exp(-\gamma m)$ where
\begin{eqnarray*}
\beta 
&=& e^{2+2\log \redund} + 2
=e^2 R^2 +2
 \leq (e^2+2) R^2 \leq 10\ \redund^2;
\\
\gamma
&\geq & 
\min\left(2t(\redund) \cdot (1+\log \redund), (\sqrt{\redund}-1)^2/8, 1/2\right) =: \gamma(\redund).
\end{eqnarray*}
As shown in Appendix~\ref{app:IncreasingFunction},  $\lim_{\redund \to \infty} \gamma(\redund) \approx 0.002$, and $\gamma(\redund) \geq 7 \cdot 10^{-6}$ when $\redund \geq 1.28$.
\end{proof}

\appendix
\section{}
\label{app:IncreasingFunction}
For $u \in (0,1)$ we have $u \log 1/u \leq e$, hence for $u\in(0,1)$ and $0<p\leq 1$:
\[
\log 1/u = (1/p) \log 1/u^p \leq (1/p) e/u^p = (1/u)^p (e/p) .
\]
Therefore, for $a>1$, $b\geq e$, $0<t<1$, using $u = t/b$, we obtain
\begin{eqnarray*}
\eta(t) 
&:=&
\sqrt{t} \cdot \left(1+a \sqrt{\log (b/t)}\right) \leq
\sqrt{t} \cdot \left( 1 + a \sqrt{(b/t)^{p}(e/p)}\right)\\
&\leq&
t^{\frac12-\frac p2} \cdot \left( t^{\frac p2} + ae^{\frac12} \sqrt{b^p/p}\right)
\leq
t^{\frac12-\frac p2} \cdot 2ae^{\frac12} \sqrt{b^p/p}
\end{eqnarray*} 
where in the last inequality we used the fact that $a e^{\frac 12} \sqrt{b^p/p} > 1$ (all the factors exceed one) 
and $t^{\frac 1p} < 1$.  For $p = 1/\log b$ we have $b^p/p = e \log b$ hence
\[
\eta(t) \leq 2 a e \sqrt{\log b} \cdot t^{\frac{1}{2} \left(1-\frac 1{\log b}\right)}
\]
The definition of $\eta(t)$ can be identified with~\eqref{eq:DefLRIPConstRandom} for $\epsilon=1$ with $t = k/m$,  $a = 2 \sqrt{2}$ and $b = eN/m =e \redund \geq e$. Denoting $c=4ae = 8\sqrt2 e$, we have just proved
\[
\eta(t) \leq (c/2) \cdot \sqrt{1 +\log \redund} \cdot t^{\frac{1}{2}\left(\frac{\log \redund}{1+\log \redund}\right)},\quad \forall 0<t<1.
\]
Defining
\begin{equation}
\label{eq:DefTRedund}
t(\redund) := \left[c^2 \cdot (1+\log \redund)\right]^{-1-\frac 1{\log \redund}} \in (0,1),
\end{equation} 
we have the guarantee $\eta\big(t(\redund)\big) \leq 1/2$ as well as the identity
\[
2 t(\redund) \cdot (1+\log \redund) 
=
2c^{-2} \cdot \left[c^2 \cdot (1+\log \redund)\right]^{-\frac{1}{\log \redund}}.
\]
The right hand side is an increasing function of $\redund$, with limit 
zero when $\redund \to 1$ and limit $2c^{-2}$ when $\redund \to \infty$. When $\redund  \geq \redund_0 := (1+4/c)^2$ we have $(\sqrt{\redund}-1)^2/8 \geq 2 c^{-2}$ hence
\[
\gamma(\redund) := \min\left(2 t(\redund) \cdot (1+\log \redund), (\sqrt \redund -1)^2/8 \right) = 2 t(\redund) \cdot (1+\log \redund).
\]
Since $c = 8\sqrt2 e = 2^{7/2} e$, we have $c^2 = 2^7 e^2$ hence 
\[
2c^{-2} = 2^{-6} e^{-2} \approx 0.0021 > 0.002,\ \mbox{and}\ \redund_0 = (1+\frac{1}{\sqrt{8}e})^2 \approx 1.277 > 1.27.
\]
For $\redund  \geq \redund_0$,
\[
\gamma(\redund) \geq 2 c^{-2} [c^2 \cdot (1+\log \redund_0)]^{-\frac1{\log \redund_0}} \approx 7.8 \cdot 10^{-6} > 7 \cdot 10^{-6}.
\]
and $\lim_{\redund \to \infty}  \gamma(\redund) = 2 c^{-2} > 2 \cdot 10^{-3}$. 
Finally, when $\redund \geq \redund_0$  we have $m(\redund) = 2/t(\redund) = 4 (1+\log \redund)/\gamma(\redund) \leq 6 \cdot 10^5 \cdot (1+\log \redund)$, and in the limit of large $\redund$ we obtain $m(\redund) \asymp 2 c^2 (1+\log \redund) \lesssim 2000 \cdot (1+\log \redund)$. 

\bibliographystyle{abbrv}
\bibliography{bernstein}

\end{document}